\documentclass[11pt,a4paper]{amsart}

\usepackage[utf8]{inputenc}
\usepackage[english]{babel}
\usepackage{graphicx}
\usepackage{tikz}
\usetikzlibrary{arrows}
\usetikzlibrary{patterns}
\usetikzlibrary{shapes}

\usepackage{latexsym}
\usepackage{xcolor}

\usepackage{mathrsfs}
\usepackage{bbm}
\usepackage[foot]{amsaddr}

\usepackage{hyperref}
\usepackage{xcolor}
\hypersetup{
    colorlinks,
    linkcolor={red!50!black},
    citecolor={blue!50!black},
    urlcolor={blue!80!black}
}
\usepackage{mathtools}
\mathtoolsset{showonlyrefs,showmanualtags}	

\usepackage{amsmath,amsfonts,amssymb}
\usepackage{lmodern}
\usepackage[left=2.8cm,right=2.8cm,top=2.8cm,bottom=2.8cm]{geometry}

\usepackage{bm}         
\DeclareMathAlphabet{\mathbbold}{U}{bbold}{m}{n}    

\usepackage{cancel}

\usepackage{comment}
\newcommand{\bee}{\begin{equation}}
\newcommand{\eee}{\end{equation}}

\newtheorem{theorem}{Theorem}
\newtheorem{lemma}[theorem]{Lemma}
\newtheorem{proposition}[theorem]{Proposition}
\newtheorem{corollary}[theorem]{Corollary}
\theoremstyle{definition}
\newtheorem{definition}[theorem]{Definition}
\theoremstyle{remark}
\newtheorem{rmk}{Remark}
\newtheorem{remark}[rmk]{Remark}

\newtheorem*{example*}{Example}

\newcommand{\distr}{D}
\newcommand{\metr}{g}
\newcommand{\g}{\gamma}

\newcommand{\beq}{\begin{equation}}
\newcommand{\eeq}{\end{equation}}
\newcommand{\R}{\mathbb{R}}

\renewcommand{\H}{\mathbb{H}}
\newcommand{\eps}{\varepsilon}

\newcommand{\lam}{\lambda}

\newcommand{\set}{\Omega}

\allowdisplaybreaks

\title[Volume of small balls in 3-dimensional contact manifolds]{Volume of small balls and sub-Riemannian curvature in 3D contact manifolds}

\date{\today}

\author{Davide Barilari}
\email{\href{mailto:davide.barilari@imj-prg.fr}{\nolinkurl{davide.barilari@imj-prg.fr}}}
\address{Univ.\ Paris Diderot, Institut de Math\'ematiques de Jussieu-Paris Rive Gauche, CNRS, Sorbonne
Universit\'e. B\^atiment Sophie-Germain, case 7012, 75205 Paris cedex 13, France}
\author{Ivan Beschastnyi}
\email{\href{mailto:i.beschastnyi@gmail.com }{\nolinkurl{i.beschastnyi@gmail.com }}}
\address{SISSA, Via Bonomea, 265, 34136 Trieste TS, Italia}
\author{Antonio Lerario}
\email{\href{mailto:lerario@sissa.it}{\nolinkurl{lerario@sissa.it}}}
\address{SISSA, Via Bonomea, 265, 34136 Trieste TS, Italia}

\begin{document}
	\begin{abstract}
		We compute the asymptotic expansion of the volume of small sub-Riemannian balls in a contact 3-dimensional manifold, and we express the first meaningful geometric coefficients in terms of geometric invariants of the sub-Riemannian structure. 
			\end{abstract}

	\maketitle
	
	\section{Introduction}
	Let $M$ be an $n$-dimensional Riemannian manifold and for $p\in M$ and $\eps>0$ let us denote by $B(p, \eps)$ the Riemannian ball of radius $\eps$ centered at $p$, i.e., the set of points in $M$ at distance at most $\eps$ from $p$.
	A classical result allows to write the asymptotic expansion of the Riemannian volume of $B(p, \eps)$ in terms of the volume $\beta_n$ of the unit ball in $\R^n$ and the scalar curvature $s(p)$ of $M$ at the point $p$:
	\beq\label{eq:smRiem}
	\mathrm{vol}(B(p,\eps))=\beta_n\,  \eps^{n}\left(1- \frac{s(p)}{6(n+2)} \eps^{2}+O(\eps^{3})\right).
	\eeq
	This formula says that to the leading order the volume of $B(p,\eps)$ coincides with the volume of the $\eps$-ball in the model space $\R^n$ and the first correction term (which is quadratic in $\eps$) depends on the curvature of $M$ at $p$. The purpose of this paper is to derive an analogue formula in the case of a $3$-dimensional contact sub-Riemannian manifold. 
	
	To be more specific, let $M$ be a $3$-dimensional manifold, $\distr\subset TM$ be a contact distribution and $g$ be a metric on $\distr$. For $p\in M$ and $\eps>0$ let us denote by $B(p, \eps)$ the \emph{sub-Riemannian} $\eps$-ball centered at $p$, i.e., the set of points in $M$ reached by a horizontal curve sorting from $p$ and of length at most $\eps.$ 
	The model space for $M$ is the $3$-dimensional Heisenberg group $\mathbb{H}^3$. 
	
	In this framework one can consider a natural volume form on $M$, called the \emph{Popp volume}, which is defined as follows. First, we write $\distr=\ker(\omega)$ for a one-form $\omega$ normalized such that $d\omega|_\distr$ coincides with the area form of $g$. If $X_{0}$ denotes the Reeb vector field of $\omega$ and $\{X_1, X_2\}$ is an oriented orthonormal basis for $\distr$, the Popp volume form is the unique $3$-form $\mu$ such that $\mu(X_1, X_2, X_{0})=1$. Notice that $\mu=\omega\wedge d\omega$. We denote by $\mathrm{vol}(A)=\int_{A}\mu$. 
	
	In the sub-Riemannian context, we obtain an expansion completely analogous to \eqref{eq:smRiem}, where now the Riemannian scalar curvature is replaced by a sub-Riemannian curvature term. On a 3-dimensional contact sub-Riemannian manifold  one can introduce two curvature functions denoted $\chi$ and $\kappa$ (see Section~\ref{s:teching}). Observe that only the latter appears in the first terms of the asymptotics of the volume of small balls.
	
	\begin{theorem}\label{t:main} Let $(M, \distr, \metr)$ be a $3$-dimensional, contact sub-Riemannian manifold and $p\in M$. As $\eps\to 0$ the Popp volume of the sub-Riemannian $\eps$-ball centered at $p$ has the following asymptotics: 
\beq\label{eq:srballs}
\mathrm{vol}(B(p,\eps))=c_0 \, \eps^{4}\left(1-c_{1} \kappa(p) \eps^{2}+O(\eps^{3})\right),\eeq
where $c_{0}$ is the volume of the unit ball in the Heisenberg group $\mathbb{H}^3$. Explicitly 
$$c_{0}=\frac{1}{12}(1+2\pi \mathrm{Si}(2\pi))\quad \textrm{and}\quad c_{1}=\frac{1}{c_0160}\left(2+4\pi \mathrm{Si}(2\pi)-\frac{1}{\pi^2}\right)>0$$ 
where
$\mathrm{Si}(x)=\int_{0}^{x}\frac{\sin t}{t}dt $ denotes the sine integral function.
\end{theorem}	
A numerical estimate of the constant appearing in the statement is given by $c_{0}\approx 0.826$ and  $c_{1}\approx 0.149$.

The expansion \eqref{eq:srballs} can be used as a definition of scalar curvature for a 3D contact sub-Riemannian structure (see Section \ref{sec:related} below).

\subsection{Connection with small time heat kernel asymptotics}
It is interesting to observe that the behavior of the volume of small balls is strictly related to the small time asymptotics of the heat kernel on the manifold. It is well-known, in fact,  that for the heat kernel $e(t,x,y)$ associated with the sub-Riemannian Laplacian, the following estimate holds for $t>0$ small enough
\begin{equation} \label{eq:gaussian}
\frac{C_{1}}{\mathrm{vol}(B(p, t^{1/2}))} \leq e(t,p,p) \leq \frac{C_{2}}{\mathrm{vol}(B(p, t^{1/2}))}.
\end{equation}
This estimate follows from more general off-diagonal Gaussian estimates on the heat kernel, investigated first in the sub-Riemannian setting in  \cite{JSC87,JSC86,SC84} and then refined in several subsequent papers in the literature. The estimate \eqref{eq:gaussian} roughly says that the main order term 
of the expansion of $\mathrm{vol}(B(p, t^{1/2}))^{-1}$ and $e(t,p,p)$ for small $t$ is the same. 

Actually, in the Riemannian case, one can observe  a stronger relationship between the two asymptotics, since the following expansion holds for $t\to 0$
\begin{equation} \label{eq:epp}
e(t,p,p)=\frac{1}{(4\pi t)^{n/2}}\left(1+\frac{s(p)}{6}t+O(t^{2})\right).
\end{equation}
This shows that even in the first order correction term the expansions \eqref{eq:smRiem} and \eqref{eq:epp} contain the same geometric invariant. 

Theorem \ref{t:main} compared with the results obtained in \cite{BAR13} confirms that the same analogies remain true in the sub-Riemannian setting, at least for 3D contact structures.

Concerning higher-dimensional structures, only partial results are known even the contact case: the small time heat kernel asymptotics for contact structures with symmetries have been obtained in \cite{BGS84,ST84} and the first coefficient has been related to the scalar Tanaka-Webster curvature. See also \cite{baudoin2013subelliptic,baudoin2014subelliptic,W16} for a recent account on heat kernel asymptotics on higher-dimensional sub-Riemannian model spaces.

\subsection{On the strategy of the proof} Even if the question of computing the asymptotics of the volume of small sub-Riemannian balls seems very natural, this is the first paper where this question is investigated in the literature. This is related to the fact that the classical ingredients are not available in sub-Riemannian geometry, as we now explain. 

A first obstruction is that the sub-Riemannian exponential map, parametrizing arclength geodesics, is defined on a non-compact set (homeomorphic to a cylinder) and is never a local diffeomorphism at zero. As a consequence, balls are not smooth, even for small radii, preventing a uniform description of the injectivity domain for the exponential map in the cotangent space. In other terms, an information on the cut locus starting from a point (that is always adjacent to the point itself) is necessary to have a correct description of balls through the exponential map.

Another obstacle in the computation of the asymptotic expansion above is that balls are not geodesically homogeneous in sub-Riemannian geometry: if one shrinks a balls to its center 
along geodesics, one does not obtain a ball of the corresponding radius. More precisely, defining $\Phi_{p,t}$ the maps that sends $x\in M$ to the point at time $t$ along the unique geodesic joining $p$ with $x$ in time 1 (this map is well defined for a.e.\ $x\in M$ on a contact sub-Riemannian manifold) one has that $\Phi_{p,t}(B(p,r))\subsetneq B(p,tr)$, with strict inclusion. It is possible actually to show that, on every 3-dimensional contact sub-Riemannian manifold, when $t\to 0$ the quantity $\mathrm{vol}(\Phi_{p,t}(B(p,r))$ goes to zero as $t^{5}$, while $\mathrm{vol}(B(p,tr))$ tends to zero as $t^{4}$. Hence, even if curvature-like invariants can be extracted by looking at the variation of a smooth volume under the geodesic flow (see \cite{ABP-distortion}), this does not permits to get the volume of small balls.

To overcome these problems we use a perturbative approach: i.e., we describe the original contact structure around a fixed point $p\in M$ as the perturbation of the Heisenberg sub-Riemannian structure, that is the metric tangent structure at a fixed point. This procedure relies on the so called nilpotent approximation of the sub-Riemannian structure and the use of a version of normal coordinates for the three dimensional exponential map developed in \cite{agrexp,ACG96,ECG96}.
This permits us to compute the asymptotic expansion of all ingredients that are involved in the computation of the volume of the ball (the exponential map, the cut time for geodesics, the Popp volume) and obtain Theorem~\ref{t:main} without explicit computations of the cut locus.

\subsection{The notion of sub-Riemannian curvature and related work}\label{sec:related}In the general case, it is not straightforward to define a notion of scalar curvature associated with a given sub-Riemannian structure.  A general approach to curvature for general sub-Riemannian strutures   has been developed in \cite{MemAMS,BR16,BR17,BR-connection}. Here a notion of generalized \emph{sectional} curvature is obtained through horizontal derivatives of the distance functions and a scalar curvature can be built by considering the trace of its horizontal part. In some specific cases, when a canonical connection (in general with non-zero torsion) associated with the metric is available,  one can also introduce curvature through this connection, as done  for instance in \cite{BG17,B-EMS, hughen, falbel}. In the 3D contact case discussed in this paper, all these approaches coincide, and define the same scalar curvature invariant $\kappa$ (up to a constant). We relate our invariants to Hughen's ones in Proposition \ref{prop:hughen} and to Falbel-Gorodski's ones in Remark \ref{rem:falbel} below.
In particular we observe that using the results of \cite{hughen} it is possible to give an alternative derivation of the asymptotic expansion of the Popp volume in exponential coordinates (see Remark \ref{rem:hughen29} below).

\subsection*{Acknowledgements}	
	This research has been supported by the ANR project SRGI ``Sub-Riemannian Geometry and Interactions", contract number ANR-15-CE40-0018. We would like to thank the anonymous referee for her/his helpful and constructive comments.	
	 \section{Technical ingredients} \label{s:teching}
A \emph{sub-Riemannian structure} on a manifold $M$ is a pair $(D,g)$ where $D$ is a vector distribution, i.e., a subbundle  of the tangent bundle $TM$, and $g$ is a smooth metric defined on $D$. It is required that $D$ is \emph{bracket generating}, i.e., Lie brackets of vector fields tangent to $D$ span the full tangent space to $M$ at every point.

Under this assumptions there is a well-defined \emph{sub-Riemannian (or Carnot-Carath\'eodory) distance} $d$, namely $d(p,q)$ is the infimum of the length of Lipschitz curves joining two points $p$ and $q$ and that are tangent to $D$ (also called horizontal curves). Here the length of the curve is computed with respect to the metric $g$. We refer to \cite{nostrolibro} for a comprehensive presentation. 

We will also say that the triplet $(M,\distr,g)$ is a sub-Riemannian manifold, when $(D,g)$ is a sub-Riemannian structure on a smooth manifold $M$.
 \subsection{Contact sub-Riemannian manifold} 
\label{s:csrm}  Let $M$ be a 3-dimensional manifold. A sub-Riemannian
structure $(D,g)$ on $M$ is said to be \emph{contact} if $\distr$ is a
contact distribution, i.e., $\distr= \ker \omega$, where $\omega\in\Lambda^{1}(M)$ satisfies $d\omega \wedge\omega \neq 0$. 

A contact distribution is bracket generating and endows $M$ with a canonical orientation. 
In what follows, if $(D,g)$ is a given sub-Riemannian structure on $M$, we always normalize the contact form $\omega$ in such a way that $d\omega \big|_{\distr}$ coincides with the Euclidean volume  defined on $\distr$ by $g$.

\begin{remark}
It is not restrictive to fix a sub-Riemannian contact structure by a pair of everywhere linearly independent vector fields  $X_{1},X_{2}$ such that $X_{1},X_{2},[X_{1},X_{2}]$ is a basis of the tangent space at every point, and declaring $X_{1},X_{2}$ to be an orthonormal frame for $g$ on the distribution $D=\mathrm{span}\{X_{1},X_{2}\}$. 
\end{remark}
The \emph{Reeb vector field} associated with the contact structure is the unique vector field $X_{0}$ satisfying 
\begin{equation}\label{eq:reeb}
X_{0}\in \ker d\omega,\qquad \omega(X_0)=1.
\end{equation} 
The Reeb vector field depends only on the sub-Riemannian structure, and its orientation.

Given an orthonormal frame $X_{1},X_{2}$ for the sub-Riemannian structure $(\distr,g)$ there exists smooth functions $c_{ij}^{k}$ defined on $M$ such that 
\begin{align} \label{eq:algebracampi}
[X_1,X_0]&=c_{01}^1 X_1+c_{01}^2 X_2, \notag\\
[X_2,X_0]&=c_{02}^1 X_1+c_{02}^2 X_2, \\
[X_2,X_1]&=c_{12}^1 X_1+c_{12}^2 X_2+X_0, \notag
\end{align}
The particular structure of the equations \eqref{eq:algebracampi} are obtained from the properties \eqref{eq:reeb} of the Reeb vector field by applying Cartan formula. In particular one can prove from $e^{tX_{0}}_{*}\distr=\distr$ that $c_{01}^{1}+c_{02}^{2}=0$

\begin{definition} We define the following quantities in terms of the structural equations \eqref{eq:algebracampi} of the orthonormal frame
\begin{itemize}
\item[(a)] the invariant $\chi$ defined by
\begin{equation} \label{eq:defchi2}
\chi=\sqrt{\frac{(c_{01}^{2}+c_{02}^{2})^{2}}{4}+(c_{01}^{1})^{2}},
\end{equation}
\item[(b)] the invariant $\kappa$ defined by
\begin{equation} \label{eq:defkappa}
\kappa=X_2(c_{12}^1)-X_1(c_{12}^2)-(c_{12}^1)^2-(c_{12}^2)^2+
\dfrac{c_{01}^2-c_{02}^1}{2}.
\end{equation}
\end{itemize}
Notice that $\chi$ and $\kappa$ are smooth functions defined on $M$.
\end{definition}
 
\begin{remark} \label{r:chipositivo} We list here some properties of the coefficients just introduced. More details are provided in \cite{nostrolibro} and \cite[Section 7.5]{MemAMS} (cf.\ also \cite{agrexp, miosr3d} and references below)
\begin{itemize}
\item[(i)] These coefficients have been first introduced in \cite{AAAICM,agrexp}. A direct calculation shows that $\chi$ and $\kappa$ are independent
of the orthonormal frame $X_1,X_2$ on the distribution and are hence local metric invariant. Indeed, Theorem \ref{t:main} also shows independently that $\kappa$ is a metric invariant.
\item[(ii)] It is possible to introduce a canonical connection $\nabla$ on a 3-dimensional contact manifold. The functions $\chi$ and $\kappa$ are expressed in terms of $\nabla$ as follows:  
\begin{equation}
\chi=\sqrt{-\mathrm{det}\,T^{\nabla}(X_{0},\cdot)},\qquad \kappa=R^{\nabla}(X_{1},X_{2},X_{2},X_{1}),
\end{equation} 
where $T^{\nabla}$ and $R^{\nabla}$ respresents the torsion and the curvature tensor associated with $\nabla$, respectively. More details on the canonical connection are provided in Appendix~\ref{s:appendix}.

\item[(iii)] One can show that $\chi\geq0$, and $\chi$ vanishes everywhere if and only if the flow of the Reeb vector field $X_{0}$ is a flow of sub-Riemannian isometries for $M$. When $\chi=0$ identically, the sub-Riemannian structure can be represented as an isoperimetric problem on a two-dimensional Riemannian manifold $N$, and $\kappa$ represents the Gaussian curvature on $N$.
\item[(iv)] The functions $\chi$ and $\kappa$ are invariant by local isometries and they are constant functions for
left-invariant structures on Lie groups. In particular if $\chi=\kappa=0$ the structure is local isometric to the Heisenberg group. See \cite{miosr3d,falbel} for the classification of left-invariant structures in terms of these invariants.
\end{itemize}
\end{remark}

	\subsection{Normal coordinates}\label{sec:coord}

The basic example of contact sub-Riemannian structure in dimension three is the Heisenberg group: this is the sub-Riemannian structure defined by the orthonormal frame in $\R^{3}$
\begin{equation} \label{eq:hhh}
\widehat{X}_{1}=\partial_{x}-\frac{y}{2}\partial_{z}, \qquad
\widehat{X}_{2}=\partial_{y}+\frac{x}{2}\partial_{z}.
\end{equation}
Notice that the normalized contact form and the corresponding Reeb vector field for this structure are 
\begin{equation}
\widehat\omega=-dz-\frac{y}{2}dx+\frac{x}{2}dy,\qquad \widehat X_{0}=-\partial_{z}.
\end{equation}

	For a general 3-dimensional contact sub-Riemannian structure, there exists a smooth normal form of the sub-Riemannian structure (i.e., of its orthonormal frame) which is the analogue of normal coordinates in Riemannian geometry. In this coordinates a general 3-dimensional contact sub-Riemannian structure is presented as  a perturbation of the Heisenberg group. 

\begin{theorem}[\cite{ACG96,ECG96}] 
\label{t:normal}
Let $M$ be a 3-dimensional contact sub-Riemannian manifold and $X_{1},X_{2}$ a local orthonormal frame. Fix $p\in M$. Then there exists a smooth coordinate system $(x,y,z)$ around $p$ such that $p=(0,0,0)$ and 
\begin{align} \label{eq:coordinate1}
X_{1}&=(\partial_{x}-\frac{y}{2}\partial_{z})+\beta y (y\partial_{x}-x\partial_{y})-\g \frac{y}{2} \partial_{z},\\[0.1cm]
X_{2}&=(\partial_{y}+\frac{x}{2}\partial_{z})-\beta x (y\partial_{x}-x\partial_{y})+\g \frac{x}{2} \partial_{z}, \label{eq:coordinate2}
\end{align}
where $\beta=\beta(x,y,z)$ and $\gamma=\gamma(x,y,z)$ are smooth functions  satisfying the following boundary conditions
\begin{equation}\label{eq:coordinate3}
\beta(0,0,z)=\g(0,0,z)=\frac{\partial \g}{\partial x}(0,0,z)=\frac{\partial \g}{\partial y}(0,0,z)=0.
\end{equation} 
\end{theorem}
Notice that, when $\beta=\gamma=0$,  one  recovers formulas \eqref{eq:hhh}.  In the same spirit as Riemannian normal coordinates, the coordinates given by Theorem~\ref{t:normal} normalize the zero-order term of the metric and have no first order correction term. For this statement to be formalized, let us introduce the notion of nilpotent approximation.

	\subsection{Nilpotent approximation}	In normal coordinates we introduce the family of dilations $\delta_\eps:\R^3\to \R^3$, for every $\eps>0$, by
	\beq \delta_\eps(x, y, z)=(\eps x, \eps y, \eps^2 z).\eeq
	For $i=1, 2$ we denote by $X_i^\eps$ the vector fields in $\R^{3}$
	\beq \label{eq:pb}
	X_i^\eps:=\eps (\delta_{\frac{1}{\eps}})_*X_i.
	\eeq
	
	For $\eps>0$, we consider the distribution $\distr^\eps=\textrm{span}\{X_1^\eps, X_2^\eps\}$; we put a metric $g^\eps$ on this distribution by declaring $\{X_1^\eps, X_2^\eps\}$ an orthonormal basis. Observe that $(\R^3, \distr^0, g^0)$ with this metric is the Heisenberg group $\H^3$. We denote by $B^\eps(1)$ the unit ball centered at the origin for the sub-Riemannian manifold $(\R^3, \distr^\eps, g^\eps)$ and by $B(\eps)\subset \R^3$  the image of $B(p, \eps)$ under the normal coordinates map. 
	
	\begin{lemma}\label{lemma:balls}For every $\eps>0$ small enough we have $B(\eps)=\delta_\eps(B^\eps(1)).$
	\end{lemma}
	\begin{proof}
It is sufficient to prove that $\gamma:I\to \R^3$ is a horizontal curve for $(\R^3, D^1, g^1)$ with length $\ell(\gamma)$  if and only if $\gamma_\eps=\delta_{1/\eps}\circ \gamma$ is a horizontal curve for $(\R^3, D^\eps, g^\eps)$ with length $\eps^{-1}\ell(\gamma)$. This is immediate from the definition \eqref{eq:pb}.
%
\end{proof}	

	Next lemma expresses the vector fields $X_i^\eps$ as perturbations of the vector fields $\widehat{X}_i$ defining the Heisenberg structure and follows from a direct computation.
	
Given a smooth function $F(x,y,z)$ of three variables, we denote by $F^{[2]}(x,y,z)$ the second order homogeneous part of its Taylor polynomial at zero. Moreover we set $F^{[2]}(x,y):=F^{[2]}(x,y,0)$. 

	\begin{lemma}\label{lemma:asymp} The following asymptotic expansion holds for $\eps\to 0$
\begin{align*}
X^{\eps}_{1}&=(\partial_{x}-\frac{y}{2}\partial_{z})-\eps^{2}\frac{y}{2}\g^{[2]}(x,y)  \partial_{z} +O(\eps^{3}),\\[0.1cm]
X^{\eps}_{2}&=(\partial_{y}+\frac{x}{2}\partial_{z})+\eps^{2}\frac{x}{2}\g^{[2]}(x,y)   \partial_{z}+O(\eps^{3}),
\end{align*}
Moreover, denoting by $X_{0}^{\eps}:=\eps^{2} (\delta_{\frac{1}{\eps}})_*X_0$ and $(c_{ij}^{k})^{\eps}$ the structure constant satisfying
$$[X^{\eps}_{j},X^{\eps}_{i}]=\sum_{k=0}^{2}(c_{ij}^{k})^{\eps}X^{\eps}_{k}, \qquad i,j=0,1,2,$$
we have
\begin{align}
(c_{12}^{1})^{\eps}&=2\eps^{2}\partial_{y}\g+ O(\eps^{3}), &
(c_{12}^{2})^{\eps}&=-2\eps^{2}\partial_{x}\g+ O(\eps^{3})\\
(c_{01}^{1})^{\eps}&=-2\eps^{2}\partial^{2}_{xy}\g + O(\eps^{3}) &
(c_{01}^{2})^{\eps}&=2\eps^{2}\partial^{2}_{x}\g + O(\eps^{3})\\
(c_{02}^{1})^{\eps}&=-2\eps^{2}\partial^{2}_{y}\g + O(\eps^{3}) &
(c_{02}^{2})^{\eps}&=2\eps^{2}\partial^{2}_{xy}\g + O(\eps^{3}).
\end{align}
where $\gamma$ is as in Theorem~\ref{t:normal} and the partial derivatives of $\gamma$ are computed at zero.
\end{lemma}
%

\begin{proof}
The expansion of $X_i^\varepsilon$ follows directly from the definitions of the vector fields and their explicit form in normal coordinates \eqref{eq:coordinate1} and \eqref{eq:coordinate3}. 

To prove the asymptotics of the structure constants, we note first that the vector fields $X_1^\varepsilon,X_2^\varepsilon$ for each $\varepsilon\geq 0$ define a contact structure with $X_0^\varepsilon$ as the Reeb field. Indeed, it is easy to verify using the definitions that 
$$
\omega^\varepsilon = \frac{1}{\varepsilon^2}\delta_\varepsilon^*\omega
$$
is the one-form defining the distribution. This means that $X_i^\varepsilon$ satisfy the structure equations \eqref{eq:algebracampi} with $(c^k_{ij})^\varepsilon$ as structure constants, for $i,j,k=0,1,2$. Knowing explicitly the Reeb field $X_0^\varepsilon$, one could expand $X_i^\varepsilon$ and $(c^k_{ij})^\varepsilon$ into power series of $\varepsilon$ and solve recursively for coefficients of $(c^k_{ij})^\varepsilon$ remembering that for the Heisenberg structure $c_{12}^{0}=-c_{21}^{0}=1$ and $c^k_{ij}= 0$ otherwise.

To avoid explicit computations, we proceed in a slightly different manner. Instead of considering our original contact structure defined by $X_i$, we look at different structure defined by $\tilde X_i$, such that the asymptotic expansions of vector fields $X_i^\varepsilon$ and $\tilde{X}_i^\varepsilon$ agree up to a certain order of $\varepsilon$. Then by repeating the previous argument we obtain an asymptotic expansion for $(\tilde{c}^k_{ij})^\varepsilon$ that agrees with asymptotic expansion for $(c^k_{ij})^\varepsilon$ up to the same order.

To prove our claim we need the asymptotic expansion up to order two. This can be achieved by considering vector fields 
\begin{align*}
\tilde X_{1}&=\partial_{x}-\frac{y(1+\eps^{2} \g^{[2]}(x,y))}{2}\partial_{z},\\[0.1cm]
\tilde X_{2}&=\partial_{y}+\frac{x(1+\eps^{2}\g^{[2]}(x,y))}{2}\partial_{z},
\end{align*}
which is just a truncation of the original vector fields. In~\cite{BAR13} explicit expressions for the corresponding one-form $\tilde{\omega}$, the Reeb vector field $\tilde{X}_0$ and the structure constants $(\tilde{c}^k_{ij})$ we found. In particular
\begin{align}
(\tilde c_{12}^{1})^{\eps}&=\frac{2\partial_{y}\g}{1+2\g}, &
(\tilde c_{12}^{2})^{\eps}&=-\frac{2\partial_{x}\g}{1+2\g}\\
(\tilde c_{01}^{1})^{\eps}&=-\frac{2((1+2\g)\partial_{xy}\g-2\partial_y\g\partial_x\g)}{(1+2\g)^2} &
(\tilde c_{01}^{2})^{\eps}&=\frac{2((1+2\g)\partial_{xx}\g-2(\partial_y\g)^2)}{(1+2\g)^2}\\
(\tilde c_{02}^{1})^{\eps}&=-\frac{2((1+2\g)\partial_{yy}\g+2(\partial_x\g)^2)}{(1+2\g)^2} &
(\tilde c_{02}^{2})^{\eps}&=\frac{2((1+2\g)\partial_{xy}\g+2\partial_y\g\partial_x\g)}{(1+2\g)^2}.
\end{align}
Lemma \ref{lemma:asymp} now is a direct consequence of the following homogeneity property. If $c^k_{ij}$ are the structure constants associated with the vector fields $X_i$, then
\begin{equation}
(c^k_{ij})^\varepsilon = \varepsilon^{d_{i}+d_{j}-d_{k}}(c^k_{ij}\circ \delta_\varepsilon)
\end{equation}
where we set $d_{1}=d_{2}=1$ and $d_{0}=2$.
\end{proof}

\subsection{Exponential map}\label{s:expmap}
In order to describe a geodesic ball, we need a good description of geodesics. As in  Riemannian geometry, one can define an analogue of the exponential map. But unlike the Riemannian case, it is defined as a map from the cotangent bundle to the manifold using Hamiltonian dynamics.

We start by defining the basis Hamiltonian $h_i: T^*M \to M$ as linear on fibers functions
$$
h_i(\lambda) = \langle \lambda, X_i\rangle, \qquad \lambda\in T^*M, \qquad i=0,1,2.
$$
We are going to use $h_i(\lambda)$ as coordinate functions on fibers of $T^*M$ and therefore from now on we do not indicate explicitly the dependence on $\lambda$.

It is well known that geodesics on a rank 2 sub-Riemannian manifold are projections of solutions of a Hamiltonian system with a quadratic Hamiltonian~\cite{nostrolibro}
\bee\label{eq:hamiltonian} H=\frac12 (h_{1}^{2}+h_{2}^{2})\eee
So we can define the exponential map $\exp:T^*M\to M$ associated to the Hamiltonian $H$ as:
\bee \exp(\lambda)=\pi(e^{\overrightarrow{H}}(\lambda)),\eee
where $\vec{H}$ is the corresponding Hamiltonian vector field. So if we wish to restrict only to geodesics that go out from a point $p\in M$, we have to consider $\lambda \in T^*_p M$, and we define the map $\exp_p :T^*_p M \to M$ as a restriction
$$
\exp_p = \exp|_{T^*_p M}.
$$

Since we are interested only in the behaviour of small balls around $p$, one can use the usual Darboux coordinates on the cotangent bundle and the standard Poisson bracket to write down explicitly the Hamiltonian system. But it is better to take a slightly more invariant approach and consider the Lie-Poisson bracket on $T^*M$. The Lie-Poisson bracket of two basis Hamiltonians $h_i,h_j$ is defined as
$$
\{h_j,h_i\}(\lambda) = \langle \lambda, [X_j,X_i] \rangle = \sum_{k=0}^{2} c^k_{ij}h_k(\lambda). 
$$ 
A bracket of any two smooth functions on the fibers of $T^*M$ can be defined via linearity and Leibnitz rule. Then our Hamiltonian system can be written as
\begin{equation} \label{eq:hamgeneral}
\begin{cases}
\dot{p}=h_{1} X_{1}(p)+h_{2} X_{2}(p)\\
\dot{h}_{i}=\{H,h_{i}\}
\end{cases}
\end{equation}

\begin{remark}
It is interesting that the invariant $\chi$ can be obtained directly from the Hamiltonian system. If we denote by
$$\{H,h_{0}\}=c_{01}^{1} h_{1}^2+(c_{01}^{2}+c_{02}^{1})h_{1}h_{2}+c_{02}^{2} h_{2}^{2}$$	
the corresponding quadratic form in $h_{1},h_{2}$, then $\mathrm{trace}\,\{H,h_{0}\}=0$ since $e^{t X_{0}}_{*}\distr=\distr$. The other invariant
$$\sqrt{-\det \{H,h_{0}\}}\geq 0$$
is exactly $\chi$ and it is zero when $(e^{t X_{0}})^{*} g=g $.
\end{remark}

It is well known that solutions of a Hamiltonian system lie on a level set of the corresponding Hamiltonian. In our case the projections of the level sets to fibers of $T^*M$ are cylinders. So we can introduce \emph{cylindrical  coordinates} $(\rho,\theta,w)$ on $T^*_p M$ as
\begin{align}
h_1 &= \rho \cos \theta \nonumber\\
h_2 &= \rho \sin \theta \label{eq:cylin_def}\\
h_0 &= -w \nonumber
\end{align}
 It follows immediately that $\rho$ is constant along solutions of \eqref{eq:hamgeneral} and it is equal to the speed of the corresponding geodesics.

We are interested in the study of small balls $B(\varepsilon)$. Lemma~\ref{lemma:balls} gives an explicit relation between $B(\varepsilon)$ and the unit ball $B^\varepsilon(1)$ of the dilated system. We can describe the ball $B^\varepsilon(1)$ and its volume using the exponential map of the dilated system, that we denote by $\exp_p^\varepsilon$. We also have a different set of cylindrical coordinates, but they are related as can be seen from the following lemma.
\begin{lemma}\label{lemma:diagram}For every $\eps>0$ let $\tau_{\eps}:T^{*}_pM\to T^*_pM$ be the map defined in cylindrical coordinates by $\tau_{\eps}(\rho, \theta, w)=\left(\rho \eps,\theta, w\right).$ Then for every $\eps>0$ the following diagram is commutative:
\beq
		\begin{tikzpicture}[xscale=5, yscale=2]

    \node (A0_0) at (0, 0) {$\R^{3}$};
    \node (A0_1) at (0, 1) {$T^*_pM$};
    \node (A1_1) at (1/2, 1) {$T^*_pM$};
    \node (A1_0) at (1/2, 0) {$\R^{3}$};
    \path (A0_1) edge [->] node [auto, swap] {$\exp_p$} (A0_0) ;
        \path (A1_1) edge [->] node [auto] {$\exp_p^\eps$} (A1_0) ;

    \path (A0_1) edge [->] node [auto] {$\tau_{\frac{1}{\eps}}$} (A1_1);
        \path (A0_0) edge [->] node [auto] {$\delta_{\frac{1}{\eps}}$} (A1_0);

      \end{tikzpicture}
      \eeq
   Here we identify $\R^{3}$ with an open neighborhood of $p$ on which normal coordinates are defined.
\end{lemma}
\begin{proof}Since both $\delta_\eps$ and $\tau_\eps$ are diffeomorphisms, we prove the equivalent statement:
\bee \delta_\eps\circ \exp_p^\eps=\exp_p\circ \,\tau_\eps.\eee
We start by recalling the definition of $\exp_p^\eps= \pi\circ e^{\overrightarrow{H^{\eps}}}$, where $H^\eps:T^*M\to \R$ is the Hamiltonian:
\bee \label{eq:hameps}H^\eps=\frac{1}{2}\left((h_1^\eps)^2+(h_2^\eps)^2\right).\eee
By definition the hamiltonians $h_i^\eps$ are given by:
\begin{align}h_i^\eps(\lambda)&=\left\langle \lambda , \eps (\delta_{1/\eps})_*X_i\right\rangle=\eps \left\langle (\delta_{1/\eps})^*\lambda , X_i\right\rangle=\eps h_i(\alpha_\eps(\lambda)),
\end{align}
where we have defined the diffeomorphisms $\alpha_\eps\doteq (\delta_{1/\eps})^*:T^*M\to T^*M$. Notice that $\alpha_\eps$ lifts $\delta_\eps:M\to M$, hence it is a symplectomorphisms. As a consequence $H^\eps=\eps^2 H\circ \alpha_\eps$. In particular we can write (we use simple identities that can be easily verified by the reader, referring for example to \cite{nostrolibro} for a detailed proof):
\begin{align}\delta_\eps\circ \exp^\eps&=\delta_\eps\circ\pi\circ e^{\eps^2\overrightarrow{H\circ \alpha_\eps}}\\
&=\delta_\eps\circ\pi\circ e^{\eps^2(\alpha_\eps^{-1})_*\overrightarrow{H}}\quad (\textrm{by \cite[Proposition 4.52]{nostrolibro}})\\
&=\delta_\eps\circ\pi\circ \alpha_\eps^{-1}\circ e^{\eps^2\overrightarrow{H}}\circ \alpha_\eps\quad (\textrm{by \cite[Lemma 2.20]{nostrolibro}})\\
&=\pi \circ e^{\eps^2\overrightarrow{H}}\circ \alpha_\eps\quad (\textrm{because $\alpha_\eps^{-1}$ lifts $\delta_{\frac{1}{\eps}}$})\\
&=\pi \circ \eps^{-2}\circ e^{\overrightarrow{H}}\circ \eps^2\circ \alpha_\eps\quad (\textrm{by \cite[Lemma 8.33]{nostrolibro}})\\
&=\pi\circ e^{\overrightarrow{H}}\circ \eps^2\circ \alpha_\eps\quad (\textrm{because $\eps^{-2}$ preserves the fibers of $\pi$}).
\end{align}
It remains to verify that $\eps^2\circ\alpha_\eps|_{T^*pM}=\tau_\eps.$ Recalling the definition of $\alpha_\eps=(\delta_{1/\eps})^*$ we see that $\alpha_\eps|_{T^*_qM}$ is given in cylindrical coordinates by:
\bee\alpha_\eps(\rho, \theta, w)=\left(\frac{\rho}{\eps}, \theta , \frac{w}{\eps^2}\right),\eee
and consequently $\eps^2\circ\alpha_\eps(\rho, \theta, w)=(\eps \rho, \theta, w)=\tau_\eps(\rho, \theta, w).$ This concludes the proof.

\end{proof}

The next proposition gives the necessary asymptotics of the Jacobian of the exponential maps $\exp^\varepsilon_p$.

\begin{proposition}\label{prop:expansion} The Jacobians $\det(J \exp^\varepsilon)$ of the family of exponential maps $\exp^\eps_p$ have the following expansion in cylindrical coordinates $(\rho,\theta,w)$ as $\eps\to 0$:
\beq \label{eq:expasymp}\det(J \exp_p^\varepsilon) = \det(J \exp_p^0) +v_2(\rho, \theta, w)\varepsilon^2 + O(\varepsilon^3),\eeq
where  $\exp_p^0:\R^3\to \R^3$ is the exponential map for the Heisenberg group and
\beq \label{eq:v2}v_2(\rho, \theta, w)= \rho^5\left(\frac{\kappa(p)}{2}g_0(w) + g_c(w)\cos 2\theta+ g_s(w)\sin 2\theta\right)\eeq
with $g_0(w),g_c(w),g_s(w)$  smooth functions of $w$ variable only. Moreover the functions $\det(J \exp_p^0) $ and $g_0$ have the following expression:
\begin{align}
\label{eq:expJ0} \det(J \exp_p^0)(\rho, \theta, w)&=\rho^3 \frac{ (2-2 \cos w-w \sin w)}{w^4}\\
\label{eq:g0}g_0(w)&= \frac{(16-3w^2)\cos w +2\cos 2w + 13w\sin w + w\sin 2w-18}{w^6}.
\end{align}
\end{proposition}

\begin{remark} 
We observe that the crucial information contained in Proposition~\ref{prop:expansion} that we use later  is 
\beq 
v_2(\rho, \theta, w)=\kappa(p)\, f_1(\rho, w)+f_2(\rho, \theta, w)\quad \textrm{with} \quad \int_{0}^{2\pi}f_2(\rho, \theta, w)d\theta=0.
\eeq 

\end{remark}  

\begin{remark}\label{rem:hughen29}As pointed out by the anonymous referee, the expansion of the Jacobian of the exponential map can also be derived from the proof of \cite[Proposition 3.6]{hughen}, which uses similar methods.
\end{remark}
\begin{proof}

We start by writing  the Hamiltonian system for the dilated structure. The Hamiltonian $H^\varepsilon$ is given by \eqref{eq:hameps} and we can write  the Hamiltonian system \eqref{eq:hamgeneral} explicitly using the Lie-Poisson bracket. We get
\begin{equation} \label{eq:ham3-dimensionalas}
\begin{cases}
\dot{p}=h^\varepsilon_{1} X_{1}(p)+h^\varepsilon_{2} X_{2}(p)\\
\dot{h}^\varepsilon_{1}=\{H^\varepsilon,h^\varepsilon_{1}\}=\{h^\varepsilon_{2},h^\varepsilon_{1}\} h^\varepsilon_{2} \\
\dot{h}^\varepsilon_{2}=\{H^\varepsilon,h^\varepsilon_{2}\} =\{h^\varepsilon_{1},h^\varepsilon_{2}\} h^\varepsilon_{1} \\
\dot{h}^\varepsilon_{0}=\{H^\varepsilon,h^\varepsilon_{0}\}= \{h^\varepsilon_{1},h^\varepsilon_{0}\} h^\varepsilon_{1}+\{h^\varepsilon_{2},h^\varepsilon_{0}\} h^\varepsilon_{2} \\
\end{cases}
\end{equation}
To rewrite our system in cylindrical coordinates, we make a change of variables on the fibers of $T^*M$
\begin{align*}
h_1^\varepsilon &= \rho \cos \theta \\
h_2^\varepsilon &= \rho \sin \theta \\
h_0^\varepsilon &= -w 
\end{align*}
We also introduce the following functions
$$
a^\varepsilon(\theta) = \frac{1}{\rho^2} \{H^\varepsilon,h_0^\varepsilon\}= (c_{01}^{1})^\varepsilon\cos(\theta)^{2}+((c_{01}^{2})^\varepsilon+(c_{02}^{1})^\varepsilon)\cos(\theta)\sin(\theta)+(c_{02}^{2})^\varepsilon \sin(\theta)^{2},
$$
$$
b^\varepsilon(\theta) =-\frac{1}{\rho} \left(\{h_1^\varepsilon,h_2^\varepsilon\}+h_0^\varepsilon\right)=(c_{12}^{1})^\varepsilon \cos \theta  + (c_{12}^{2})^\varepsilon\sin \theta.
$$
Then, after various simplifications, we obtain the Hamiltonian system
\begin{equation} \label{eq:ham3-dimensionalas2}
\begin{cases}
\dot{p}=\rho\cos(\theta)  X_1^\varepsilon(p) + \rho\sin(\theta)  X_2^\varepsilon(p)\\
\dot{\rho}=0 \\
\dot{\theta}= w-\rho b^\varepsilon(\theta)\\
\dot{w}= -\rho^{2} a^\varepsilon(\theta).
\end{cases}
\end{equation}

Now we expand the right-hand side and the phase variables in series of powers of $\varepsilon$. This will give us a number of ordinary differential equations on the coefficients, that we are going to solve. Since the Hamiltonian system \eqref{eq:ham3-dimensionalas} is smooth, depends smoothly on $\varepsilon$, and we are interested only in the behaviour for small $\varepsilon$, the resulting asymptotics is going to be uniform.

We fix normal coordinates $(x,y,z)$ around $p\in M$. In this coordinates $p=(0,0,0)$. Then we fix an initial covector $(\bar\rho,\bar \theta,\bar w)$
and look at how the corresponding geodesic changes as $\varepsilon$ goes to zero. Thus our asymptotic expansions are
\begin{equation}
\begin{cases}
x(t) = x_0(t) + x_1(t)\varepsilon + x_2(t)\frac{\varepsilon^2}{2} + O(\varepsilon^3) \\[0.2cm]
y(t) = y_0(t) + y_1(t)\varepsilon + y_2(t)\frac{\varepsilon^2}{2} + O(\varepsilon^3) \\[0.2cm]
z(t) = z_0(t) + z_1(t)\varepsilon + z_2(t)\frac{\varepsilon^2}{2} + O(\varepsilon^3) 
\end{cases}\quad
\begin{cases}
w(t) = w_0(t) + w_1(t)\varepsilon + w_2(t)\frac{\varepsilon^2}{2} + O(\varepsilon^3) \\[0.2cm]
\theta(t) = \theta_0(t) + \theta_1(t)\varepsilon + \theta_2(t)\frac{\varepsilon^2}{2} + O(\varepsilon^3)
\end{cases}
\end{equation}
Since the initial covector $(\bar\rho,\bar \theta,\bar w)$ is independent of $\varepsilon$, we  have the following boundary conditions
\begin{align}
x_i(0) &= y_i(0) = z_i(0) = 0, & &\forall i \in \mathbb{N}_0 \\
\theta_i(0) &= w_i(0) = 0, & &\forall i \in \mathbb{N}\\
\theta_0(0) &= \bar \theta, \ w_0(0) = \bar w.
\end{align}
Let us look at the principal and first order terms of the asymptotics. First of all we note that from Lemma~\ref{lemma:asymp} it follows that all the structure constants are $O(\varepsilon^2)$. Thus functions $a^\varepsilon$ and $b^\varepsilon$ are $O(\varepsilon^2)$ as well. Using the asymptotics of $X^\varepsilon_i$ from the same lemma, we then obtain a system for the zero-order term
\begin{equation*}
\begin{cases}
\dot{x}_0=\bar \rho\cos\theta_0\\
\dot{y}_0=\bar \rho\sin\theta_0\\
\dot{z}_0= \frac{ \bar \rho}{2}(x_0  \sin\theta_0 -  y_0 \cos\theta_0 )
\end{cases}\qquad
\begin{cases}
\dot{\theta}_0= w_0\\
\dot{w}_0= 0
\end{cases}
\end{equation*}
But this is nothing but the geodesic equations on the Heisenberg group whose solutions are explicit
\begin{equation}
\label{eq:heisexp}
\begin{cases}
x_0(t)=\frac{\bar\rho(\sin(\bar wt+\bar\theta)-\sin \bar\theta)}{\bar w}\\
y_0(t)=-\frac{\bar\rho(\cos(\bar wt+\bar\theta)-\cos \bar\theta)}{\bar w}\\
z_0(t)=\frac{\bar\rho^2(\bar wt-\sin t \bar w)}{2\bar w^2}
\end{cases}\qquad
\begin{cases}\theta_0(t)= \bar wt+\bar \theta\\
w_0(t)= \bar w
\end{cases}
\end{equation}
Thus we see that as $\varepsilon \to 0$ geodesics of the dilated system converge to the geodesics of the Heisenberg group as expected. Moreover, setting $t=1$ in \eqref{eq:heisexp} and differentitating with respect to $(\bar \rho, \bar \theta, \bar w)$ we immediately obtain \eqref{eq:expJ0}.

Next we write the system of order one. We obtain
\begin{equation*}
\begin{cases}
\dot{x}_1=-\bar \rho\theta_1\sin\theta_0\\
\dot{y}_1=\bar \rho\theta_1\cos\theta_0\\
\dot{z}_1=\frac{\bar \rho}{2}(-y_1\cos\theta_0 + x_1 \sin\theta_0 + x_0\theta_1\cos\theta_0 + y_0\theta_1\sin\theta_0 )
\end{cases}\qquad
\begin{cases}
\dot{\theta}_1= w_1\\
\dot{w}_1= 0
\end{cases}
\end{equation*}
Using the zero boundary conditions we get
$$
w_1(t) = \theta_1(t) = x_1(t) = y_1(t) = z_1(t) =0, \qquad \forall t.
$$
From here it immediately follows that the zero-order term in the expression is the Heisenberg term and the first order term is identically zero.

We continue this procedure. At each next step we integrate expression involving only terms from the previous steps. Then we can plug all asymptotic expansions into the Jacobian and after various simplifications, we obtain the result. 
The second order term in the asymptotics of the exponential map is a result of similar but rather long computations. The simplification of the expression for the Jacobian becomes a tedious exercise after applying various trigonometric identities. 
\end{proof}

 \subsection{The Popp volume and curvature invariants in normal coordinates}
On a contact sub-Riemannian manifold it is possible to define a canonical volume that depends only on the sub-Riemannian structure, called Popp volume. Here we recall its construction only in the 3-dimensional case, the interested reader is referred to \cite{montgomery} and \cite{BRpopp} for the general construction and its explicit expression in terms of an adapted frame.

Given an orthonormal frame $X_{1},X_{2}$ for the sub-Riemannian structure and the corresponding Reeb vector field $X_{0}$, let us denote by $\nu_{1},\nu_{2},\nu_{0}$ the dual basis of 1-forms. The Popp volume $\mu$ is defined as the three-form $\mu=\pm \nu_{1}\wedge\nu_{2}\wedge \nu_{0}$. The sign is chosen in such a way that the volume is positive.
 
Recall that we denote by $F^{[2]}(x,y,z)$ the second order homogeneous part of a smooth function $F(x,y,z)$ of three variables and $F^{[2]}(x,y):=F^{[2]}(x,y,0)$.

 \begin{lemma}\label{lemma:gamma}Using normal coordinates (introduced in Section \ref{sec:coord}) the Popp volume form can be written as $\mu=\psi \, dx\wedge dy\wedge dz$, where $\psi:\R^3\to \R$ is a smooth function such that:
 \beq\label{eq:ga2} \psi (x, y,z)=1-2\gamma^{[2]}(x,y)+O\left(\|(x, y, z)\|^3\right).\eeq
 where
 \begin{equation} \label{eq:pdgamma}
\gamma^{[2]}(x,y)=x^{2} \partial^{2}_{x}\gamma +2xy \partial_{xy}^{2}\gamma+y^{2}\partial^{2}_{y}\gamma
\end{equation}
where the partial derivatives of $\gamma$ in \eqref{eq:pdgamma} are computed at zero. 
 \end{lemma}

\begin{proof} For notational convenience, let us introduce $X_{3}:=[X_{2},X_{1}]$ and denote by $(x_{1},x_{2},x_{3})$ the coordinates $(x,y,z)$. Let $\nu_{1},\nu_{2}, \nu_{3}$ be the dual basis of 1-forms to $X_{1},X_{2},X_{3}$. Notice that the Popp volume $\mu$ is written as $\mu=\nu_{1}\wedge\nu_{2}\wedge \nu_{3}$ (up to the choice of positive sign), as a consequence of the relation
$[X_2,X_1]=X_0 \mod \distr$ (cf.~\eqref{eq:algebracampi}).
Considering the coordinate expression of the vector fields and the basis of 1-forms 
$$X_{i}=\sum_{j=1}^{3}a_{ij}\partial_{j},\qquad i=1,2,3,\qquad
\textrm{and}
\qquad \nu_{k}= \sum_{l=1}^{3}b_{kl} dx_{l},\qquad k=1,2,3,$$
for some smooth functions $a_{ij}, b_{kl}$. Then the matrices $A=(a_{ij})$ and $B=(b_{kl})$ satisfy the relation $B=(A^{T})^{-1}$. 
In particular 
\begin{equation}\label{eq:muuu}
\mu =|\det(B)|dx\wedge dy\wedge dz=|\det(A)|^{-1}dx\wedge dy\wedge dz.
\end{equation}
From the explicit expression of the vector fields \eqref{eq:coordinate1}-\eqref{eq:coordinate2} and boundary conditions \eqref{eq:coordinate3}, it is easy to check that every coefficient of the vector fields containing $\beta$ gives a contribution of order at least three  in the expansion of the determinant. Hence, to compute the expansion of \eqref{eq:muuu} up to second order,  it is not restrictive to assume that $\beta=0$. Under this assumption, one computes 
\begin{equation}
[X_{1},X_{2}]=\left(1+ \gamma+\frac{1}{2}(x\partial_{x}\gamma+y\partial_{y}\gamma)\right)\partial_{z},
\end{equation}
which implies 
\begin{align*}
|\det(A)|&=1+\gamma(x,y,z)+\frac{1}{2}(x\partial_{x}\gamma(x,y,z)+y\partial_{y}\gamma(x,y,z))+ O\left(\|(x, y, z)\|^3\right)\\
&=1+2\gamma^{[2]}(x,y)+ O\left(\|(x, y, z)\|^3\right),
\end{align*}
where in the last equality  we used the boundary conditions \eqref{eq:coordinate3}. Taking the inverse and combining with \eqref{eq:muuu}, the proof is completed.
\end{proof}

Along the same lines of the proof of Lemma \ref{lemma:gamma} one obtains the following result. A proof is contained in \cite[Lemma 4]{BAR13}.
\begin{lemma} In normal coordinates (introduced in Section \ref{sec:coord}) writing 
$$\gamma^{[2]}(x,y)=a x^{2}  +2bxy+cy^{2},$$ we have the following expression for the curvature-like invariants at the origin
\beq 
\kappa(p)=2(a+c), \qquad \chi(p)=2\sqrt{b^{2}+(c-a)^{2}}.
\eeq

\end{lemma}

	\subsection{Cut-time asymptotic} 
	Let $\gamma:[0,T]\to M$ be a horizontal curve. We say that $\g$ is a \emph{length-minimizer} if $d(\gamma(0),\gamma(T))=\ell(\gamma)$. Notice that this notion is independent on the parametrization of the curve.

Fix now a horizontal curve  $\gamma:[0,T]\to M$ parametrized by arc-length. Then we define
\begin{equation}
t_{cut}(\gamma)=\sup\{t>0: \gamma|_{[0,t]} \text{ is a length-minimizer}\}.
\end{equation}
If $\gamma$ is a geodesic parametrized by arclength on a 3-dimensional contact manifold, then it is well-known that $t_{cut}(\gamma)>0$. This is related with the fact that there are no abnormal minimizers, see for instance \cite[Chapter 8]{nostrolibro} and \cite[Appendix]{BRinterpolation}.

Parametrizing geodesics as in Section~\ref{s:expmap} with covectors in cylindrical coordinates $(\rho,\theta,w)$ we have a well-defined cut time associated with every initial covector with $\rho=1$.

We give here the asymptotic expansion for the cut time of geodesics. This result is obtained combining \cite[Theorem 4.2]{agrexp} and \cite[Theorem 5.2]{agrexp}, covering the case $\chi(p)\neq 0$ and $\chi(p)=0$, respectively.\footnote{A note on the reference: to recover the cut time, denoted $\ell_{*}(\theta;\nu)$ in \cite[Theorem 5.2]{agrexp}, one needs the formula of $\ell_{1}(\theta;\nu)$ of \cite[Theorem 5.1]{agrexp}, whose expression contains a typo. Indeed the second summand of its expression is $-\pi\kappa(q_{0})\nu^{-3}$ (and not $-\pi\kappa(q_{0})\nu^{-2}$) as it can be directly checked from the proof.}
\begin{proposition} \label{t:cut3-dimensional}
We have the following asymptotic expansion for the cut time from $p\in M$. 
\begin{equation}\label{eq:tcut}
t_{cut}(1,\theta,w)= \frac{2\pi}{|w|}-\frac{\pi(\kappa(p)+2\chi(p) \sin^{2}\theta)}{|w|^{3}}+O\left(\frac{1}{|w|^{4}}\right), \qquad w \to \pm\infty.
\end{equation}
\end{proposition}

Thanks to this result we get an asymptotic description of the set of initial parameters mapped on the ball of radius $\eps$ through the exponential map.	
\begin{corollary}\label{cor:D}
For $\eps>0$ small enough there exists a set $\set(\eps)\subset T^*_pM$, whose description in cylindrical coordinates is given by
	\beq \label{eq:crucial}
	\set(\eps)=\left\{|\rho|\leq \eps,\, \theta\in[0, 2\pi],\, w\in[-2\pi+\rho^2f(\theta)+O(\rho^3), 2\pi-\rho^2f(\theta)+O(\rho^3)]\right\},\eeq
	where $f(\theta)$ is a smooth function of $\theta$.  Moreover the following properties holds:
	\begin{itemize}
	\item[(i)] \label{eq:D1} $\overline B(p,\eps)=\exp(\set(\eps))$;
	\item[(ii)] \label{eq:D2} $\exp|_{\mathrm{int}(\set(\eps)\backslash\{\rho=0\})}$ is injective with injective differential;
	\item[(iii)] \label{eq:D3} $\exp(\mathrm{int}(\set(\eps)\backslash\{\rho=0\}))$ has full measure in $B(p,\eps).$
	\end{itemize}
	\end{corollary}
We observe that the existence of a set $\Omega(\eps)$ satisfying conditions (i)-(iii) above is true as soon as the sub-Riemannian structure does not contain non-trivial abnormal minimizers. This condition is, in particular, satisfied for contact sub-Riemannian manifolds. A sample image of $\Omega(\varepsilon)$ is presented in Figure~\ref{fig:egg}. The crucial fact in Corollary~\ref{cor:D} is the asymptotic description given in \eqref{eq:crucial}.

\begin{figure}[t]
\centering
\includegraphics[scale=0.5]{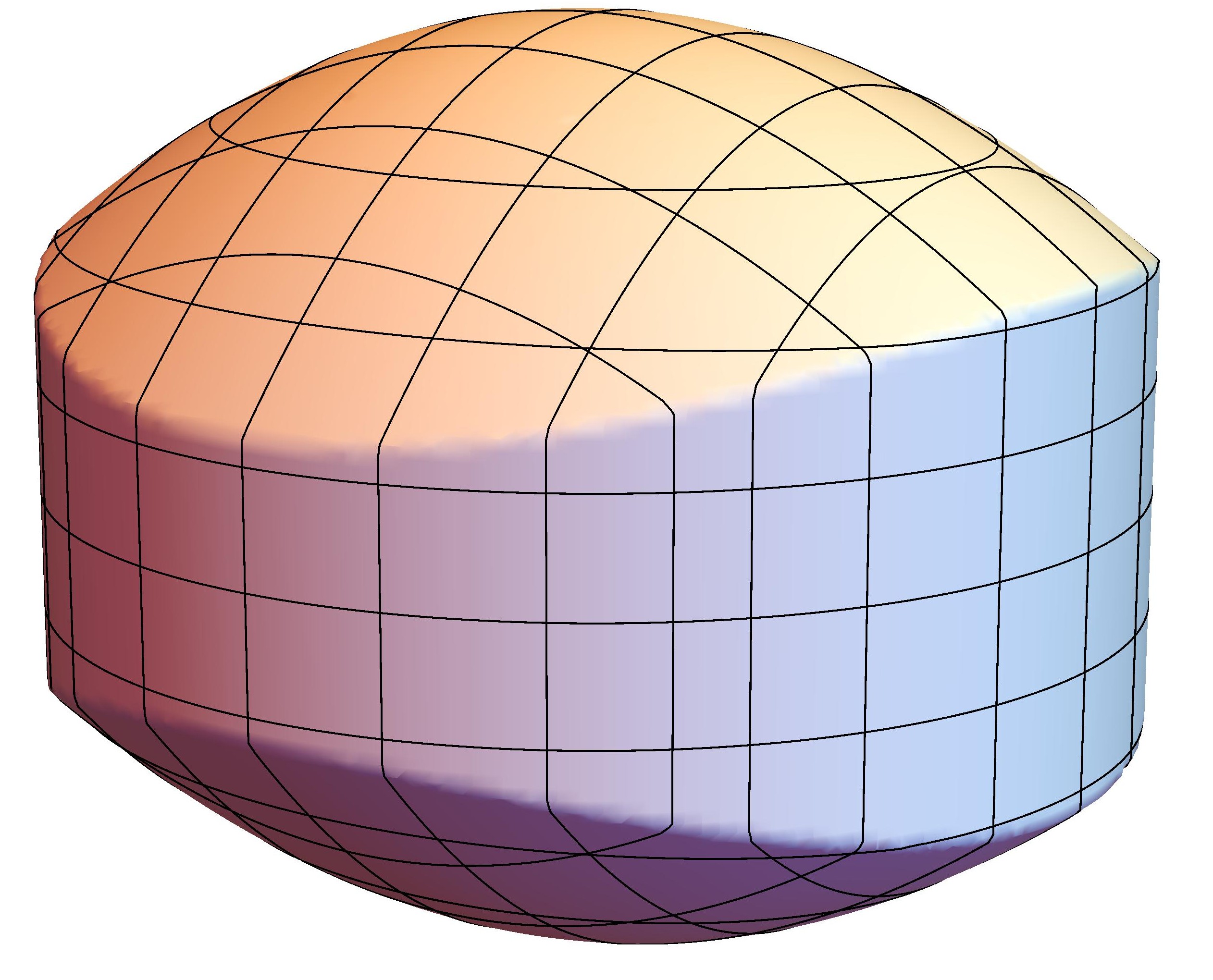}
\label{fig:egg}
\caption{A picture of the region $\Omega(\varepsilon)$}
\end{figure}

\begin{proof} First we discuss the existence of the set $\Omega(\eps)$ satisfying properties (i)-(iii). For $\eps>0$ small enough, the closure of the ball $B(p,\eps)$ is compact, hence for every point $x\in B(p,\eps)$ there exists a length-minimizer $\gamma_{\lambda} :[0,1] \to M$, associated with initial covector $\lambda =(\rho,\theta,w) \in T_p^*M$ joining $p$ with $x$. Recall that contact sub-Riemannian structures have no non-trivial abnormal minimizers, thus $\gamma_\lambda$ does not contain any abnormal segment and cannot be minimizing after its first conjugate time. Moreover, under the assumption that there are no non-trivial abnormal minimizers, a cut time is either the first conjugate time or a point where two optimal geodesics intersect. For a proof of these statements one can see \cite[Chapter 3, Chapter 8]{nostrolibro} or \cite[Appendix A]{BRinterpolation}.

For every unit initial covector $\lambda=(1,\theta,w)$ we have  $\gamma_{\lambda}(t)= \exp_{p}(t(1,\theta,w))=\exp_{p}(t,\theta,tw)$ and this trajectory is by definition a length-minimizer  up to the  corresponding cut time $t_{cut}(1,\theta,w)$. Notice that $\exp_{p}(0,\theta,w)=p$ for every $\theta,w$. We stress that for every $s\in ]0,t_{cut}(\lam)[$, the restriction $\gamma_{\lam}|_{[0,s]}$ is a length-minimizer hence does not contain neither cut points (by construction), nor conjugate points (by length-minimality). 

Let us introduce the star-shaped set in $ T^{*}_{p}M$
$$A=\{(s,\theta,sw)\in T^{*}_{p}M \mid \theta\in [0,2\pi], w\in \R, 0\leq s\leq t_{cut}(1,\theta,w)\},$$
and set $\Omega(\eps):=A\cap \{0\leq \rho\leq \eps\}$.
It follows by construction that $\exp_p : \Omega(\eps) \to \overline B(p,\eps)$ is onto. Moreover 
$$\exp_p (\mathrm{int} ( \Omega(\eps)\setminus\{\rho=0\})) = B(p,\eps)\setminus \mathrm{Cut}(p),$$ which has full measure in $B(p,\eps)$.
The fact that $\exp_p$ is injective with injective differential on the open set $\mathrm{int} ( \Omega(\eps)\setminus\{\rho=0\})$, is a consequence of the fact that length minimizers do not contain neither cut nor conjugate points.

To complete the proof of the statement, we compute the asymptotic description of the set $\Omega(\eps)$ in the cotangent space $T^{*}_{p}M$.
Let us rewrite the set $A$ as follows, in cylindrical coordinates\footnote{One can use the following identity, for $\varphi:[0,+\infty)\to [0,+\infty)$ an invertible function
$$\{(s,sw)\in \R^{2}\mid  w\in \R, 0\leq s \leq \varphi(|w|)\}=\{(x,y)\in \R^{2}\mid x\geq 0 ,  |y|\leq x\varphi^{-1}(x)\}.$$}
\begin{equation}\label{eq:disegnata}
A=\{(\rho,\theta,w)\in T^{*}_{p}M \mid \theta\in [0,2\pi], \rho\geq 0,  |w|\leq  \rho t^{-1}_{cut}(1,\theta,\rho)\},
\end{equation}
where $t^{-1}_{cut}(1,\theta,\cdot)$ means the inverse function of the map $w\mapsto t_{cut}(1,\theta,w)$, for a fixed $\theta \in [0,2\pi]$. Notice that the function $t_{cut}(1,\bar{\theta},\cdot)$ is smooth at infinity, for fixed $\bar \theta \in [0,2\pi]$, with derivative approaching a positive constant, and therefore it is invertible close to infinity. 

The expansion of $t^{-1}_{cut}(1,\theta,\cdot)$ at zero is then obtained from the one of $t_{cut}(1,\theta,\cdot)$ at infinity \eqref{eq:tcut}  as follows
\begin{equation} \label{eq:tmeno1}
t^{-1}_{cut}(1,\theta,\rho)=\frac{2\pi}{\rho} - \frac{(\kappa +2\chi\sin^2\theta)}{4\pi}\rho + O(\rho^2), \qquad \rho \to 0.
\end{equation}
Multiplying \eqref{eq:tmeno1} by $\rho$ and combining with \eqref{eq:disegnata}, one gets the statement by setting 
\begin{equation}
f(\theta)=\frac{\kappa +2\chi\sin^2\theta}{4\pi}
\qedhere
\end{equation} 
%
\end{proof}
	\section{Proof of main theorem}
	We compute the volume of the ball $B(p, \eps)$ in normal coordinates. Recall that in this coordinate chart the ball is denoted simply $B(\eps)$ and the Popp volume writes $\mu=\psi \,dx\,dy\,dz$. We have  	\begin{align} \mathrm{vol}(B(p, \eps))&=\int_{B(\eps)}\psi (x, y, z)\,dx\,dy\,dz\\
	&=\eps^4\int_{B^\eps(1)}\psi(\eps x, \eps y, \eps^2z) \,dx\,dy\,dz \quad \textrm{(using Lemma \ref{lemma:balls})}\\
	&=\eps^4\left(\int_{B^\eps(1)}(1-2\eps^2\gamma^{[2]}(x,y, z)) \,dx\,dy\,dz +O(\eps^3)\right)\quad \textrm{(using Lemma \ref{lemma:gamma})}
	\end{align}
Using again Lemma \ref{lemma:balls}, we can write
	\begin{align}B^\eps(1)&=\delta_{\frac{1}{\eps}}(B(\eps))\\
	&=\delta_{\frac{1}{\eps}}(\exp(\set(\eps))\quad (\textrm{by property (i) from Corollary \ref{cor:D}})\\
	&=\exp^\eps(\tau_{\frac{1}{\eps}}(\set(\eps))) \quad (\textrm{by Lemma \ref{lemma:diagram}})
	\end{align}
	Observe also that $\set^\eps(1)\doteq\tau_{\frac{1}{\eps}}(\set(\eps))$ has the following description in cylindrical coordinates:
	\beq \tau_{\frac{1}{\eps}}(\set(\eps))=\{|\rho|<1, \theta \in [0, 2\pi], w\in[-2\pi+\eps^2\rho^2f(\theta)+O(\eps^3), 2\pi-\eps^2\rho^2f(\theta)+O(\eps^3)].\eeq
	In particular we can write:
	\begin{align}\label{eq:int1}
	 \mathrm{vol}(B(p, \eps))&=\eps^4\left(\int_{\exp^\eps(\set^\eps(1))}(1-2\eps^2\gamma^{[2]}(x,y,z)) \,dx\,dy\,dz +O(\eps^3)\right).\\
	\end{align}
	Observe now that properties (ii) and (iii) from Corollary \ref{cor:D} remains true if we compose the various maps with a diffeomorphim, after considering the images of the corresponding sets under the diffeomorphism itself. In particular, since both $\delta_{\frac{1}{\eps}}$ and $\tau_{\frac{1}{\eps}}$ are diffeomorphisms, we can apply the change of variable formula and compute the integral in \eqref{eq:int1} as:
	\begin{align}
	\label{eq:int2}
	 \mathrm{vol}(B(p, \eps))=\eps^4\left(\int_{0}^1\int_{0}^{2\pi}\int_{-2\pi+\eps^2\rho^2f(\theta)+O(\eps^3)}^{2\pi-\eps^2\rho^2f(\theta)+O(\eps^3)}u(\rho, \theta, w)\,dw \, d\theta\, d\rho +O(\eps^3) \right),
	\end{align}
	where
	\beq u(\rho, \theta, w)=(1-2\eps^2\gamma^{[2]}(\exp^\eps(\rho, \theta, w))|\det (J\exp^\eps)(\rho, \theta, w)|.\eeq
	We compute now the expansion in $\eps$ of the various terms involved. Let us start with $u$, which using the expansion $\exp_{p}^\eps=\exp_p^0+O(\varepsilon)$ and \eqref{eq:expasymp} we can write as
	\begin{align}u(\rho, \theta, w)=& \det (J \exp_p^0)(\rho, \theta, w)+\\
	&+\eps^2\left(-2\gamma^{[2]}(\exp_p^0(\rho, \theta, w)) \det (J \exp_p^0)(\rho, \theta, w)+v_2(\rho, \theta, w)\right)+O(\eps^3)\\
	=&\,u_0(\rho, \theta, w)+\eps^2u_2(\rho, \theta, w)+O(\eps^3)
	\end{align}
	Observe now also that: 
	\begin{align}
	\int_{-2\pi+\eps^2\rho^2f(\theta)+O(\eps^3)}^{2\pi-\eps^2\rho^2f(\theta)+O(\eps^3)}u(\rho, \theta, w)dw=&
	\int_{-2\pi+\eps^2\rho^2f(\theta)+O(\eps^3)}^{2\pi-\eps^2\rho^2f(\theta)+O(\eps^3)}u_0(\rho, \theta, w)+\eps^2u_2(\rho, \theta, w)dw +O(\eps^3)\\
	=&\int_{-2\pi}^{2\pi}u_0(\rho, \theta, w)+\eps^2 u_2(\rho, \theta, w)dw+\\
	&-2 \eps^{2}\rho^2f(\theta)\left(u_0(\rho, \theta, -2\pi)+u_0(\rho, \theta, 2\pi)\right)+O(\eps^3)\\
	\label{eq:int3}=&\int_{-2\pi}^{2\pi}u_0(\rho, \theta, w)dw+\eps^2 \int_{-2\pi}^{2\pi}u_2(\rho, \theta, w)dw+O(\eps^3),
	\end{align}
	where in the last line we have used the crucial fact that $u_0(\rho, \theta, 2\pi)=u_0(\rho, \theta, -2\pi)=0$, as it can be immediately verified from \eqref{eq:expJ0}. Recall that $u_0=\det(J\exp_p^0)$ is the Jacobian determinant in the Heisenberg group. In more geometric terms, the last equality is saying that the cut time coincides also with the first conjugate time in the Heisenberg group.
	
	Consider now the fixed domain $\set=\{|\rho|\leq 1,\, \theta\in [0, 2\pi], \,w\in [-2\pi, 2\pi]\}$. Plugging \eqref{eq:int3} into \eqref{eq:int2}, we obtain:
	\beq\label{eq:int4} 
	 \mathrm{vol}(B(p, \eps))=\eps^4\left(\int_{\set} u_0+\eps^2\int_{\set}u_2+O(\eps^3)\right).\eeq
	From the definition of $u_0=\det(J\exp_p^0)$, we immediately recognize 
	\beq 
	c_{0}:=\int_{\set} u_0=\textrm{volume of the unit ball in the Heisenberg group}.\eeq
	For the integral of $u_2$, we proceed analyzing the various functions appearing in its definition
	\begin{align} u_2&=-2\gamma^{[2]}(\exp_p^0(\rho, \theta, w)) \det (J \exp_p^0)(\rho, \theta, w)+v_2(\rho, \theta, w).\end{align}
	Writing $\gamma^{[2]}(x,y, z)=ax^2+2bxy+cy^2$ as in \eqref{eq:ga2} and using \eqref{eq:heisexp}, we have: 
	\begin{align}-2\gamma^{[2]}(\exp_p^0) 
	\det (J \exp_p^0)
	&=4(a+c)\frac{\rho^5\sin\left(\frac{w}{2}\right)^2\left(2\cos w+w\sin w-2\right)}{w^6}+\\
	& \quad -\frac{4 \rho^2 \sin\left(\frac{w}{2}\right)^2 \left((a-c) \cos(2 \theta+w)+b \sin(2 \theta+w)\right)}{w^2}\, \det(J\exp_p^0)\\
	\label{eq:g22}&=2\kappa(p)\, \frac{\rho^5\sin\left(\frac{w}{2}\right)^2\left(2\cos w+w\sin w-2\right)}{w^6}+\\
	&\quad  +\cos(2 \theta+w)g_1(\rho, w)+\sin(2 \theta+w)g_2(\rho, w),\end{align}
	where in the last line we have used the fact that $2(a+c)=\kappa(p)$ and that $\det(J\exp_p^0)$ only depends on $(\rho, w)$ (see the explicit expression \eqref{eq:expJ0}).
	Note in particular that, exchanging the order of integration and using the fact that for every fixed $w\in [0, 2\pi]$ the integrals $\int_{0}^{2\pi}\cos(2\theta+w)d\theta$ and $\int_{0}^{2\pi}\sin(2\theta+w)d\theta$ vanish, \eqref{eq:g22} implies:
	\begin{align} \int_\set -2\gamma^{[2]}(\exp_p^0)\det(J\exp_p^0)&=2\kappa(p)\, \int_{0}^1\int_{-2\pi}^{2\pi}2\pi\frac{\rho^5\sin\left(\frac{w}{2}\right)^2\left(2\cos w+w\sin w-2\right)}{w^6}dw\,d\rho\\
	\label{eq:int5}&= \kappa(p)\, \int_{-2\pi}^{2\pi} \frac{2\pi}{3}\frac{\sin\left(\frac{w}{2}\right)^2\left(2\cos w+w\sin w-2\right)}{w^6}dw.
	\end{align}
	Let us look now at the integral of the function $v_2$. Using its explicit expression \label{eq:v2} and integrating the $\theta$-variable first, we obtain:
	\begin{align} \int_\set v_2&= \int_\set  \rho^5\left(\kappa(p)\frac{g_0(w)}{2} + g_c(w)\cos 2\theta+ g_s(w)\sin 2\theta\right)\,dw \, d\theta\, d\rho\\
	\label{int:g01}&=\kappa(p)\int_{0}^1\int_{-2\pi}^{2\pi}\rho^5 2\pi\frac{g_0(w)}{2}dw d\rho\\
	&=\kappa(p)\, \int_{-2\pi}^{2\pi}\frac{\pi}{6}g_0(w)dw.
	\end{align}
	Combining \eqref{eq:int5} and \eqref{int:g01} we obtain:
	\begin{align}\int_{\set}u_2&=\kappa(p)\,\int_{-2\pi}^{2\pi}\left( \frac{2\pi}{3}\frac{\sin\left(\frac{w}{2}\right)^2\left(2\cos w+w\sin w-2\right)}{w^6}+\frac{\pi}{6}g_0(w)\right) dw\\
	&=\kappa(p)\int_{-2\pi}^{2\pi} \frac{\pi}{2}\frac{(5w\sin w-(w^2-8)\cos w-8)}{w^6}dw\\
	&= \kappa(p)\frac{1}{160}\left(\frac{1}{\pi^2}-2-4\pi \mathrm{Si}(2\pi)\right).
	\end{align}
	Together with \eqref{eq:int4} this finally gives:
	\beq
	\textrm{vol}(B(p, \eps))=\eps^4c_0\left(1-\kappa(p)c_1\eps^2+O(\eps^3)\right),\eeq
	where:
	\beq c_1=\frac{1}{c_0160}\left(2+4\pi \mathrm{Si}(2\pi)-\frac{1}{\pi^2}\right)>0\quad \textrm{and}\quad c_0=\frac{1}{12}(1+2\pi \mathrm{Si}(2\pi)).\eeq
The explicit formula of $c_{0}$, which is the volume of the unit ball in the Heisenberg group, coincides witht the one obtained in \cite[Remark 39]{ABB-Hausdorff}.

\appendix
\newcommand{\Tor}{T}
\section{Remarks on curvature coefficients} \label{s:appendix}

The study of complete sets of invariants, connected with the problem of equivalence of 3D sub-Riemannian contact structures, has been previously considered in the literature in different context and with different languages, as for instance in \cite{hughen} and \cite{falbel}. 

In this appendix we recall the relation of the geometric invariants $\chi$ and $\kappa$ defined in Section~\ref{s:csrm}, with those used in \cite{hughen,falbel}. Notice that $\chi$ and $\kappa$ does not give a complete sets of invariants since there exists two (left-invariant) non-isometric sub-Riemannian structures with same $\chi$ and $\kappa$. An explicit formula for the isometry is given cf.\ \cite{miosr3d} (see also \cite[Remark 3.1]{falbel}).

\subsection{Invariants of a canonical connection} We extend the sub-Riemannian metric $g$ on $\distr$ to a global Riemannian structure (that we denote with the same symbol $g$) by promoting $X_0$ to an unit vector orthogonal to $\distr$.

We define the \emph{contact endomorphism} $J:TM\to TM$ by:
\begin{equation}
g(X,JY)=d\omega(X,Y),\qquad \forall X,Y\in \Gamma(TM). 
\end{equation}
Clearly $J$ is skew-symmetric w.r.t.\ to $g$. In the 3-dimensional  case, the previous condition forces $J^{2}=-\mathbb{I}$ on $\distr$ and $J(X_{0})=0$.

\begin{theorem}[canonical connection, \cite{blair,tanno89,falbel}]
There exists a unique linear connection $\nabla$  on $(M,\omega,g,J)$ such that
\begin{itemize}
\item[(i)] $\nabla\omega = 0$,
\item[(ii)] $\nabla X_0 = 0$,
\item[(iii)] $\nabla g = 0$,
\item[(iv)] $\Tor(X,Y) = d\omega(X,Y) X_0$ for any $X,Y \in \Gamma(\distr)$,
\item[(v)] $\Tor(X_0,JX) = -J \Tor(X_0,X)$ for any vector field $X \in \Gamma(TM)$,
\end{itemize}
where $\Tor$ is the torsion tensor of $\nabla$.
\end{theorem}
If $X$ is a horizontal vector field, so is $\Tor(X_0,X)$. As a consequence, if we define $\tau(X) = \Tor(X_0,X)$, $\tau$ is a symmetric horizontal endomorphism which satisfies $\tau \circ J + J \circ \tau = 0$, by property (v). Notice that $\mathrm{trace}(\tau)=0$ and $\det (\tau)\leq 0$.

A standard computation gives the following result. 
\begin{lemma}\label{l:krrr}
Let $R^{\nabla}$ be the curvature associated with the  connection $\nabla$. Then
\begin{equation} \label{eq:krrr}
\kappa=R^{\nabla}(X_{1},X_{2},X_{2},X_{1}),\qquad \chi=\sqrt{-\mathrm{det}(\tau)}.
\end{equation}
 \end{lemma}
Notice that a contact structure is $K$-type if and only if $X_0$ is a Killing vector field or, equivalently, if and only if $\tau =0$.

	\subsection{Relation with other invariants in the literature}

Let us denote by $g$ the Riemannian metric on $M$ obtained by declaring the Reeb vector field $X_{0}$ to be orthogonal to the distribution and of unit norm and denote by $\overline \nabla$ the Levi-Civita connection associated with the Riemannian metric $g$. The Christoffel symbols $\overline \Gamma_{ij}^{k}$ of this connections are defined by
\begin{equation}\label{eq:covariant}
\overline \nabla_{X_{i}}X_{j}=\overline \Gamma_{ij}^{k}X_{k}, \qquad \forall', i,j=0,1,2,
\end{equation} 
and related with the structural functions of the frame by the following formulae:
\begin{equation}
\overline \Gamma_{ij}^{k}=-\frac{1}{2}(c_{ij}^{k}-c_{jk}^{i}+c_{ki}^{j}).
\end{equation}
Let us denote by $\mathrm{Sec}(\Pi_x)$ the sectional curvature with respect to $\overline \nabla$ of the plane $\Pi_x$ generated by two vectors $v,w\in T_{x}M$.
\begin{proposition} \label{p:secR}
The sectional curvature of the plane $\Pi_x = \distr_x$ is
\begin{equation}\label{eq:sec}
\mathrm{Sec}(\distr_{x})=\kappa+\chi^{2}-\frac34.
\end{equation}
\end{proposition}
\begin{proof}
It is a long but straightforward computation, using the explicit expression of the covariant derivatives \eqref{eq:covariant}. In terms of an orthonormal frame $X_{1},X_{2}$ for the distribution $\distr_{x}$ we have 
\begin{align*}
\mathrm{Sec}(\distr_{x})&=g(\nabla_{X_{1}}\nabla_{X_{2}}X_{2}-\nabla_{X_{2}}\nabla_{X_{1}}X_{2}
-\nabla_{[X_{1},X_{2}]}X_{2},X_{1})\\
&=-X_{1}(c_{12}^{2})+X_{2}(c_{12}^{1})-(c_{12}^{1})^{2}-(c_{12}^{2})^{2}+\frac{1}{2}(c_{01}^{2}-c_{02}^{1})
+(c_{01}^1)^{2}+\frac{1}{4}(c_{02}^{1}+c_{01}^{2})^{2}-\frac{3}{4},
\end{align*}
and \eqref{eq:sec} follows from the explicit expressions \eqref{eq:defchi2} and \eqref{eq:defkappa} of $\chi$ and $\kappa$.
\end{proof}
In \cite{hughen}, using the Cartan's moving frame method, Hughen introduces the family of generating invariants $a_{1},a_{2}, K \in C^\infty(M)$. 
\begin{proposition}[Relation with invariants defined by Hughen]\label{prop:hughen}
\label{p:hughen}
We have the following identity
\begin{equation}\label{eq:secsec}
\kappa=K, \qquad \chi =\sqrt{a_{1}^{2}+a_{2}^{2}}.
\end{equation}
\end{proposition}
\begin{proof}

The author in \cite[p.15]{hughen} proves that $K=4W$, where $W$ is the Tanaka-Webster curvature of the CR structure associated with the sub-Riemannian one. Notice that also that $\kappa=4W$ from Lemma~\ref{l:krrr}, hence $\kappa=K$.   Moreover one has \cite[p.15]{hughen} 
\begin{equation}
\mathrm{Sec}(\distr_{x})=K+a_{1}^{2}+a_{2}^{2}-\frac34.
\end{equation}
This, together with Proposition \ref{p:secR}, gives the other relation 
$ \chi^{2} =a_{1}^{2}+a_{2}^{2}.$
\end{proof}
	
\begin{remark}[Relation with invariants defined by Falbel-Gorodski]\label{rem:falbel}
In \cite{falbel}, the authors introduce a family of generating invariants $K,\tau_{0},W_{1},W_{2}$, associated with this connection. It follows directly from Lemma~\ref{l:krrr} that $\kappa=K$ and $\chi=\tau_{0}$.
\end{remark}	
	\bibliographystyle{abbrv}
	\bibliography{biblio}	
\end{document}